\documentclass[11pt]{article}
\usepackage[a4paper]{anysize}\marginsize{1.5cm}{1.5cm}{1.5cm}{1cm}
\pdfpagewidth=\paperwidth \pdfpageheight=\paperheight
\usepackage{amsfonts,amssymb,amsthm,amsmath,eucal,tabu}
\usepackage{pgf}
 \usepackage{array}
 \usepackage{pstricks}
 \usepackage{pstricks-add}
 \usepackage{pgf,tikz}
 \usetikzlibrary{automata}
 \usetikzlibrary{arrows}
 \usepackage{indentfirst}
\theoremstyle{plain}
\newtheorem{thm}{Theorem}[section]
\newtheorem{theorem}[thm]{Theorem}

\newtheorem*{theoremA}{Main Theorem}

\newtheorem{proposition}[thm]{Proposition}

\newtheorem{conjecture}[thm]{Conjecture}

\theoremstyle{definition}
\newtheorem{definition}[thm]{Definition}

\newtheorem{example}[thm]{Example}

\newtheorem{thevarthm}[thm]{\varthmname}

\newenvironment{varthm*}[1]{\trivlist\item[]{\bf #1.}\it}{\endtrivlist}

\renewcommand\geq{\geqslant}

\renewcommand\leq{\leqslant}

\newcommand\be{\begin{eqnarray*}}
\newcommand\ee{\end{eqnarray*}}

\newcommand\call{{\mathcal L}}

\newcommand\newop[2]{\def#1{\mathop{\rm #2}\nolimits}}
\newop\log{log}
\newop\ord{ord}
\newop\Gal{Gal}
\newop\SL{SL}
\newop\Bl{Bl}
\newop\mult{mult}
\newop\mass{mass}
\newop\div{div}
\newop\codim{codim}
\newop\sing{sing}
\newop\vdim{vdim}
\newop\edim{edim}
\newop\Ass{Ass}
\newop\size{size}
\newop\reg{reg}
\newop\satdeg{satdeg}
\newop\supp{supp}
\newop\Neg{Neg}
\newop\Nef{Nef}
\newop\Nefh{Nef_H}
\newop\Eff{Eff}
\newop\Zar{Zar}
\newop\MB{MB}
\newop\MBxC{MB\mathit{(x,C)}}
\newop\NnB{NnB}
\newop\Bigg{Big}
\newop\Effbar{\overline{\Eff}}

\def\keywordname{{\bfseries Keywords}}%
\def\keywords#1{\par\addvspace\medskipamount{\rightskip=0pt plus1cm
\def\and{\ifhmode\unskip\nobreak\fi\ $\cdot$
}\noindent\keywordname\enspace\ignorespaces#1\par}}
\def\subclassname{{\bfseries Mathematics Subject Classification
(2000)}\enspace}
\def\subclass#1{\par\addvspace\medskipamount{\rightskip=0pt plus1cm
\def\and{\ifhmode\unskip\nobreak\fi\ $\cdot$
}\noindent\subclassname\ignorespaces#1\par}}

\begin{document}
\title{Harbourne constants and arrangements of lines on smooth hypersurfaces in $\mathbb{P}^3_{\mathbb{C}}$}
\author{Piotr Pokora}
\date{\today}
\maketitle
\thispagestyle{empty}
\begin{abstract}
In this note we find a bound for the so-called linear Harbourne constants for smooth hypersurfaces in $\mathbb{P}^{3}_{\mathbb{C}}$.
\subclass{14C20, 14J70}
\end{abstract}

%*****************************************************************************
\section{Introduction}

   In this short note we find a global estimate for Harbourne constants which were introduced in \cite{BdRHHLPSz} in order to capture and measure
   the bounded negativity on various birational models of an algebraic surface.
\begin{definition}
   Let $X$ be a smooth projective surface. We say that $X$ \emph{has bounded negativity}
   if there exists an integer $b(X)$ such that for every \emph{reduced} curve $C \subset X$ one has the bound
   $$C^{2} \geq -b(X).$$
\end{definition}
   The bounded negativity conjecture (BNC for short) is one of the most intriguing problems in the theory of projective surfaces and
   attracts currently a lot of attention, see \cite{Duke, BdRHHLPSz, Harbourne1, XR}. It can be formulated as follows.
\begin{conjecture}[BNC]
   An arbitrary smooth \emph{complex} projective surface has bounded negativity.
\end{conjecture}
Some surfaces are known to have bounded negativity
(see \cite{Duke, Harbourne1}). For example, surfaces with $\mathbb{Q}$-effective anticanonical divisor such as Del Pezzo surfaces, K3 surfaces and Enriques surfaces have bounded negativity. However, when we replace these surfaces by their blow ups, we do not know if bounded negativity is preserved. Specifically, it is not known whether the blow up of $\mathbb{P}^{2}$ at ten general points has bounded negativity or not.

   Recently in \cite{BdRHHLPSz} the authors have showed the following theorem.
\begin{theorem}(\cite[Theorem~3.3]{BdRHHLPSz})
   Let $\mathcal{L}$ be a line configuration on $\mathbb{P}^{2}_{\mathbb{C}}$. Let $f: X_{s} \rightarrow \mathbb{P}^{2}_{\mathbb{C}}$ be the blowing up at $s$ distinct points on $\mathbb{P}^{2}_{\mathbb{C}}$ and let $\widetilde{\call}$ be the strict transform of $\mathcal{L}$. Then we have $\widetilde{\call}^{2} \geq -4\cdot s$.
\end{theorem}
   In this note, we generalize this result to the case of line configurations on smooth hypersurfaces $S_{n}$ of degree $n \geq 3$ in $\mathbb{P}^{3}_{\mathbb{C}}$.

   A classical result tells us that every smooth hypersurface of degree $n=3$ contains $27$ lines. For smooth hypersurfaces of degree $n=4$ we know that the upper bound of the number of lines on quartic surfaces is 64 (claimed by Segre \cite{Segre} and correctly proved by Sch\"utt and Rams \cite{SR}). In general, for degree $n \geq 3$ hypersurfaces $S_{n}$ Boissi\'ere and Sarti (see \cite[Proposition~6.2]{BS}) showed that the number of lines on $S_{n}$ is less than or equal to $n(7n-12)$.

Using  techniques similar to the one introduced in \cite{BdRHHLPSz} we prove the following result.
\begin{theoremA} Let $S_{n}$ be a smooth hypersurface of degree $n \geq 4$ in $\mathbb{P}^{3}_{\mathbb{C}}$. Let $\mathcal{L} \subset S_{n}$ be a line configuration, with the singular locus ${\rm Sing}(\mathcal{L})$ consisting of $s$ distinct points. Let $f : X_{s} \rightarrow S_{n} $ be the blowing up at ${\rm Sing}(\mathcal{L})$ and denote by $\widetilde{\call}$ the strict transform of $\mathcal{L}$. Then we have $$\widetilde{\call}^2 > -4s -2n(n-1)^2.$$
\end{theoremA}
In the last part we study some line configurations on smooth complex cubics and quartics in detail. Similar systematic studies on line configurations on the projective plane were initiated in \cite{Szpond}.

\section{Bounded Negativity viewed by Harbourne Constants}
We start with introducing the Harbourne constants \cite{BdRHHLPSz}.

\begin{definition}\label{def:H-constants}
   Let $X$ be a smooth projective surface and let
   $\mathcal{P}=\{ P_{1},\ldots,P_{s} \}$ be a set of mutually
   distinct $s \geq 1$ points in $X$. Then the \emph{local Harbourne constant of $X$ at $\mathcal{P}$}
   is defined as
   \begin{equation}\label{eq:H-const for calp}
      H(X;\mathcal{P}):= \inf_{C} \frac{\left(f^*C-\sum_{i=1}^s \mult_{P_i}C\cdot E_i\right)^2}{s},
   \end{equation}
   where $f: Y \to X$ is the blow-up of $X$ at the set $\mathcal{P}$
   with exceptional divisors $E_{1},\ldots,E_s$ and the infimum is taken
   over all \emph{reduced} curves $C\subset X$.\\
   Similarly, we define the \emph{$s$--tuple Harbourne constant of $X$}
   as
   $$H(X;s):=\inf_{\mathcal{P}}H(X;\mathcal{P}),$$
   where the infimum now is taken over all $s$--tuples of mutually
   distinct points in $X$. \\
   Finally, we define the \emph{global Harbourne constant of $X$} as
   $$H(X):=\inf_{s \geq 1}H(X;s).$$
\end{definition}
   The relation between Harbourne constants and the BNC can be expressed in the following way.
   Suppose that $H(X)$ is a finite real number.
   Then for any $s \geq 1$ and any reduced curve $D$ on the blow-up of $X$
   at $s$ points, we have
   $$D^2 \geq sH(X).$$
   Hence the BNC holds on all blow ups of $X$ at $s$ mutually distinct points with the constant $b(X) = sH(X)$.
   On the other hand, even if $H(X)=-\infty$, the BNC might still be true.

   It is very hard to compute Harbourne constants in general. Moreover, it is quite tricky to find these numbers even for the simplest types of reduced curves on a well-understood surface.
   \section{Proof of the main result}
Given a configurations of lines on $S_{n}$ we denote by $t_{r}$ the number of its $r$-ple points, at which exactly $r$ lines of the configuration meet. In the sequel we will repeatedly use two elementary equalities, namely $\sum_{i} {\rm mult}_{P_{i}}(C) = \sum_{k \geq 2}kt_{k}$ and $\sum_{k\geq 2} t_{k} = s$. In this section we will study \emph{linear Harbourne constants} $H_{L}$. We define only the local linear Harbourne constant for $S_{n}$ containing a line configuration $\mathcal{L}$ since this is the only one difference comparing to Definition \ref{def:H-constants}.
\begin{definition}
Let $S_{n}$ be a smooth hypersurface of degree $n \geq 2$ in $\mathbb{P}^{3}_{\mathbb{C}}$ containing at least one line and let
   $\mathcal{P}=\{ P_{1},\ldots,P_{s} \}$ be a set of mutually
   distinct $s$ points in $S_{n}$. Then the \emph{local linear Harbourne constant of $S_{n}$ at $\mathcal{P}$}
   is defined as
   \begin{equation}
      H_{L}(S_{n}; \mathcal{P}):= \inf_{\mathcal{L}} \frac{\widetilde{\call}^{2}}{s},
   \end{equation}
where $\widetilde{\call}$ is the strict transform of $\mathcal{L}$ with respect to the blow up $f : X_{s} \rightarrow S_{n} $ at $\mathcal{P}$
   and the infimum is taken over all \emph{reduced} line configurations $\mathcal{L} \subset S_{n}$.
\end{definition}

 Our proof is based on the following result due to Miyaoka \cite[Section~2.4]{Miyaoka}.
 \begin{theorem}
 Let $S_{n}$ be a smooth hypersurface in $\mathbb{P}^{3}_{\mathbb{C}}$ of degree $n \geq 4$ containing a configuration of $d$ lines. Then one has
 $$nd -t_{2} + \sum_{k \geq 3}(k-4)t_{k} \leq 2n(n-1)^2.$$
 \end{theorem}
Now we are ready to give a proof of the Main Theorem.
\begin{proof}
Pick a number $n \geq 4$. Recall that using the adjunction formulae one can compute the self-intersection number of a line $l$ on $S_{n}$, which is equal to
$$l^{2} = -2 - K_{S_{n}}.l = -2 - \mathcal{O}(n-4).l = 2-n.$$
Observe that the local linear Harbourne constant at ${\rm Sing}(\mathcal{L})$ has the following form
\begin{equation}
\label{Hconst}
H_{L}(S_{n}; {\rm Sing}(\mathcal{L})) = \frac{ (2-n)d + I_{d} - \sum_{k\geq 2}k^{2}t_{k}}{\sum_{k\geq 2}t_{k}},
\end{equation}
where $I_{d} = 2 \sum_{i < j} l_{i}l_{j}$ denotes the number of incidences of $d$ lines $l_{1}, ..., l_{d}$.
It is easy to see that we have the combinatorial equality
$$I_{d} = \sum_{k \geq 2}(k^{2} - k)t_{k},$$
hence we obtain
$$I_{d} -\sum_{k \geq 2}k^{2}t_{k} = -\sum_{k \geq 2}kt_{k}.$$
Applying this to (\ref{Hconst}) we get
$$H_{L}(S_{n}; {\rm Sing}(\mathcal{L})) = \frac{ (2-n)d - \sum_{k \geq 2} kt_{k}}{\sum_{k \geq 2}t_{k}}.$$
Simple manipulations on the Miyaoka inequality lead to
$$nd + t_{2} -4\sum_{k \geq 2}t_{k} - 2n(n-1)^2 \leq -\sum_{k\geq 2} kt_{k},$$
and finally we obtain
$$H_{L}(S_{n}; {\rm Sing}(\mathcal{L})) \geq -4 + \frac{ 2d + t_{2} -2n(n-1)^2}{s},$$
which completes the proof.
\end{proof}
It is an interesting question how the linear Harbourne constant behaves when degree $n$ of a hypersurface grows. We present two extreme examples.
\begin{example}
Let us consider the Fermat hypersurface of degree $n \geq 3$ in $\mathbb{P}^{3}_{\mathbb{C}}$, which is given by the equation
$$F_{n} \,\, : \,\, x^{n} + y^{n} + z^{n} + w^{n} = 0.$$
It is a classical result that on $F_{n}$ there exists the line configuration $\mathcal{L}_{n}$ consisting of $3n^{2}$ lines and delivers $3n^{3}$ double points and $6n$ points of multiplicity $n$. It is easy to check that
$${\rm lim}_{n \rightarrow \infty} H_{L}(F_{n}; {\rm Sing}(\mathcal{L}_{n})) = {\rm  lim}_{n \rightarrow \infty} \frac{ 3 n^{2} \cdot (2-n) + 12n^{3} - 6n^{2} - 4 \cdot 3n^3 - n^{2} \cdot 6n} {3 n^3 + 6n} = -3.$$
On the other hand, the Main Theorem gives
$$ {\rm lim}_{n \rightarrow \infty} H_{L}(F_{n}; {\rm Sing}(\mathcal{L}_{n})) \geq -4 + {\rm lim}_{n \rightarrow \infty} \frac{ 6n^{2} + 3n^3 - 2n(n-1)^2}{3n^3 + 6n} = -3 \frac{2}{3},$$
which shows that the estimate given there is quite efficient.
\end{example}
\begin{example}
\label{Rams}
This construction comes from \cite{R}. Let us consider Rams hypersurface $\mathbb{P}^{3}_{\mathbb{C}}$ of degree $n \geq 6$ given by the equation
$$R_{n} \, : \, x^{n-1}\cdot y + y^{n-1} \cdot z + z^{n-1} \cdot w + w^{n-1}\cdot x = 0.$$
 On $R_{n}$ there exists a configuration $\mathcal{L}_{n}$ of $n(n-2)+4$ lines, which delivers exactly $2n^2 - 4n + 4$ double points -- this configuration is the grid of $n(n-2)+2$ vertical disjoint lines intersected by two horizontal disjoint lines. The local linear Harbourne constant at ${\rm Sing}(\mathcal{L}_{n})$ is equal to
$$H_{L}(R_{n}; {\rm Sing}(\mathcal{L}_{n})) = \frac{ (n^2-2n+4)\cdot(2-n) + 4n^2 - 8n + 8 - 4\cdot(2n^2 - 4n + 4)}{2n^2 - 4n + 4} = \frac{-n^3}{2n^2 -4n + 4}.$$
Then ${\rm lim}_{n \rightarrow \infty} H_{L}(R_{n}; {\rm Sing}(\mathcal{L}_{n})) = - \infty.$
\end{example}
Example \ref{Rams} presents a quite interesting phenomenon since we can obtain very low linear Harbourne constants having singularities of minimal orders  -- the whole game is made by the large number of (disjoint) lines.
\section{Smooth cubics and quartics}
We start with the case $n = 3$. As we mentioned in the first section every smooth cubic surface contains $27$ line, and the configuration of these lines have only double and triple points. These triple points are called \emph{Eckardt points}. Now we find a lower bound for the linear Harbourne constant for such hypersurfaces.
\begin{proposition}
Under the above notation one has
$$H_{L}(S_{3}; {\rm Sing}(\mathcal{L})) \geq -2 \frac{5}{11}.$$
\end{proposition}
\begin{proof}
Recall that the combinatorial equality \cite[Example II.20.]{Urzua} for cubic surfaces has the form
$$135 = t_{2} + 3t_{3}.$$
 Moreover, another classical result asserts that the maximal number of \emph{Eckardt} points is equal to $18$ and this number is obtained on Fermat cubic. In order to get a sharp lower bound for $H_{L}$ we need to consider the case when the number of Eckardt points is the largest. To see this we show that the linear Harbourne constant for $t$ triple points is greater then for $t+1$ triple points. Simple computations show that
$$H_{L}(S_{3};t) = \frac{-297+3t}{135-2t},$$
$$H_{L}(S_{3};t+1) = \frac{-294+3t}{133-2t},$$
and $H_{L}(S_{3};t+1) < H_{L}(S_{3};t)$ iff $(-297+3t)\cdot(133-2t) - (-294+3t)\cdot(135-2t) > 0$ for all $t \in \{0, ..., 18\}$, which is obvious to check.
Having this information in hand we can calculate that for $18$ triple points and $81$ double points the local linear Harbourne constant at ${\rm Sing}(\mathcal{L})$ is equal to
$$H_{L}(F_{3};{\rm Sing}(\mathcal{L})) = \frac{27\cdot(-1) + 270 - 4\cdot81 - 9 \cdot 18}{99}  = -2 \frac{5}{11},$$
which ends the proof.
\end{proof}
\begin{example}
Now we consider the case $n=4$ and we start with the configuration of $64$ lines on Schur quartic $Sch$. It is well-known that every line from this configuration intersects exactly $18$ other lines -- see for instance \cite[Proposition 7.1]{SR}. One can check that these $64$ lines deliver $8$ quadruple points, $64$ triple points and $336$ double points (in \cite{Urzua} we can find that the number of double points is equal to $192$, which is false). Then the local linear Harbourne constant at ${\rm Sing}(\mathcal{L})$ is equal to
$$H_{L}(Sch; {\rm Sing}(\mathcal{L})) = \frac{(-2)\cdot64 + 1152 -16\cdot8 - 9\cdot 64 - 4\cdot 336}{336 + 64 + 8} = -2.509.$$
\end{example}
Now we present an example of a line configuration on a smooth quartic which deliver the most negative (according to our knowledge) local linear Harbourne constant for this kind of surfaces.
\begin{example}[Bauer configuration of lines]
Let us consider Fermat quartic $F_{4}$. It is well-known that on $F_{4}$ there exists the configuration of $48$ lines. From this configuration one can extract a subconfiguration of $16$ lines which has only $8$ quadruple points.
Then the local linear Harbourne constant at ${\rm Sing}(\mathcal{L})$ is equal to
$$H_{L}(F_{4}; {\rm Sing}(\mathcal{L})) = \frac{ 16\cdot(-2) + 16\cdot6  - 16 \cdot8}{8} = -8.$$

Using Main Theorem we get $H_{L}(F_{4}; {\rm Sing}(\mathcal{L})) \geq -9$, which also shows efficiency of our result.
\end{example}
\paragraph*{\emph{Acknowledgement.}}
The author like to express his gratitude to Thomas Bauer for sharing Example 4.3, to S\l awomir Rams for pointing out his construction in \cite{R} and to Tomasz Szemberg and Halszka Tutaj-Gasi\'nska for useful remarks. Finally, the author would like to thank the anonymous referee for many useful comments which allowed to improve the exposition of this note. The author is partially supported by National Science Centre Poland Grant 2014/15/N/ST1/02102.
%*****************************************************************************

%***************************************************************************** % Addresses

\bigskip
   Piotr Pokora,
   Instytut Matematyki,
   Pedagogical University of Cracow,
   Podchor\c a\.zych 2,
   PL-30-084 Krak\'ow, Poland.

\nopagebreak
   \textit{E-mail address:} \texttt{piotrpkr@gmail.com}

%*****************************************************************************

\end{document}